\documentclass[final,3p,times]{elsarticle}
    \usepackage{hyperref}
    \usepackage{amsmath}
    \usepackage{amssymb}
    \usepackage{url}
    \usepackage{amsthm}
    \usepackage{mathtools}
    \usepackage{xcolor}
\usepackage{subfig}

\usepackage[colorinlistoftodos,prependcaption,textsize=tiny]{todonotes}

\usepackage[linesnumbered, ruled]{algorithm2e}		

    \newcommand{\R}[0]{{\mathbb{R}}}
    \newcommand{\C}[0]{{\mathbb{C}}}
    \newcommand{\IR}[0]{{\mathbb{IR}}}
    \newcommand{\IC}[0]{{\mathbb{IC}}}

    \def\tluste#1{\protect{\textrm{\boldmath $#1$}}}
    \newcommand{\omace}[1]{\mbox{$\overline{{#1}}$}}	
    \newcommand{\umace}[1]{\mbox{$\underline{{#1}}$}}  
    \newcommand{\imace}[1]{\tluste{#1}} 		
\newcommand{\smace}[1]{\tluste{#1}{}^S} 
    \newcommand{\ivr}[1]{\tluste{#1}} 		
    \newcommand{\inum}[1]{\tluste{#1}} 		
    \newcommand{\mmid}[0]{;\ }		
    \newcommand{\ihull}[0]{{\square}}	
    \newcommand{\hulli}[1]{\big[#1\big]}	
    \def\Mid#1{#1^c}		
    \def\Rad#1{#1^\Delta}		
    \newcommand{\mna}[1]{{\mathcal{#1}}}

\renewcommand\Re{\operatorname{\textsf{Re}}}
\renewcommand\Im{\operatorname{\textsf{Im}}}
\newcommand{\seznam}[1]{{\{1, \ldots, {#1}\}}}

\DeclareMathOperator{\Mag}{mag}	
\DeclareMathOperator{\diag}{diag}

\def\clqq{``}
\def\crqq{''}
\def\quo#1{\clqq{}#1\crqq{}}  

    \newtheorem{theorem}{Theorem}
    
    \newtheorem{proposition}[theorem]{Proposition}
    \newtheorem{observation}[theorem]{Observation}

    \newtheorem{example}[theorem]{Example}

\makeatletter
\def\ps@pprintTitle{%
 \let\@oddhead\@empty
 \let\@evenhead\@empty
 \def\@oddfoot{}%
 \let\@evenfoot\@oddfoot}
\makeatother    

\begin{document}

    \begin{frontmatter}

    \title{Computing the spectral decomposition of interval matrices and a study on interval matrix powers}

    \author[kam,cas]{David Hartman\corref{mycorrespondingauthor}}
    \cortext[mycorrespondingauthor]{Corresponding author}
    \ead{hartman@kam.mff.cuni.cz}

    \author[kam]{Milan Hlad\'{i}k}
    \ead{hladik@kam.mff.cuni.cz}

    \author[kam]{David \v{R}\'{i}ha}
    
    \journal{}

\address[kam]{Charles University, Faculty  of  Mathematics  and  Physics, Department of Applied Mathematics,\\ Malostransk\'e n\'am.~25, 11800, Prague, Czech Republic.
}
\address[cas]{The Institute of Computer Science of the Czech Academy of Sciences, Prague.}

\begin{abstract}
We present an algorithm for computing a spectral decomposition of an interval matrix as an enclosure of spectral decompositions of particular realizations of interval matrices. The algorithm relies on tight outer estimations of eigenvalues and eigenvectors of corresponding interval matrices. We present a method for general interval matrices as well as its modification for symmetric interval matrices. As an illustration, we apply the spectral decomposition to computing powers of interval matrices. Numerical results suggest that a simple binary exponentiation is more efficient for smaller exponents, but our approach becomes better when computing higher powers or powers of a special type of matrices. In particular, we consider symmetric interval and circulant interval matrices. In both cases we utilize some properties of the corresponding classes of matrices to make the power computation more efficient.
\end{abstract}


    \end{frontmatter}



\section{Introduction}

Eigenvalue decomposition of a real or complex matrix is, no doubts, a very important tool in numerical analysis and algorithm design. In practice, however, the data are usually estimated or measured with some errors. Uncertainty in data can be modelled by intervals representing the range of possible values, giving rise to the concept of interval matrices~\cite{MooKea2009, Neu1990}. This motivated us to extend spectral decomposition to interval matrices.

In the literature, we can found many works on enclosing or approximating eigenvalues of interval matrices and on the related topics dealing with stability of interval matrices. On the other hand, much less results exist on bounding eigenvectors or considering the whole eigenvalue decomposition. In this paper, we investigate eigenvalue decomposition for both general interval matrices and symmetric interval matrices. Before we formulate the problem, we need to introduce some notation first.

\subsection{Interval notation}

An interval matrix is the set of matrices
$$
\imace{A}=\{A\in\R^{m\times n}\mmid \umace{A}\leq A \leq \omace{A}\},
$$
where $\umace{A},\omace{A}\in\R^{m\times n}$ are given lower and upper bound matrices, respectively, and inequalities are understood entrywise. Interval vectors are regarded as one column interval matrices and similarly for interval scalars. The set of $m$-by-$n$ interval matrices is denoted by $\IR^{m\times n}$.
The midpoint and radius matrices associated to $\imace{A}$ are defined respectively as
$$
\Mid{A}=\frac{1}{2}(\umace{A}+\omace{A}),\quad
\Rad{A}=\frac{1}{2}(\omace{A}-\umace{A}).
$$
An (interval) enclosure of a bounded and nonempty set $\mna{S}\subset\R^n$ is any interval vector containing~$\mna{S}$, that is, $\ivr{v}\in\IR^n: \mna{S}\subseteq\ivr{v}$. The smallest interval enclosure is called an interval hull and denoted
$$
\ihull\mna{S}=\bigcap_{\ivr{v}\in\IR^n:\mna{S}\subseteq \ivr{v}}\ivr{v}.
$$
The magnitude of an interval $\inum{a}\in\IR$ is the largest absolute value of the values from the interval, i.e., 
$\Mag(\inum{a})=\max_{a\in\inum{a}}{|a|}=|\Mid{a}|+\Rad{a}$. 
For interval arithmetic, see, e.g., the textbooks \cite{MooKea2009,Neu1990}. 

Let $\imace{A}\in\IR^{m\times n}$ and $\ivr{b}\in\IR^m$. The solution set of the system of interval equations $\imace{A}x=\ivr{b}$ is defined as the union of all solutions of all instances, that is,
$$
\{x\in\R^n\mmid \exists A\in\imace{A},\,\exists b\in\ivr{b}: Ax=b\}.
$$
There are many methods known for computing an enclosure of the solution set; see \cite{MooKea2009,Neu1990,Hla2014b,Rum2010,Ska2018} among others.

Let $\imace{A}\in\IR^{n\times n}$ such that both $\Mid{A}$ and $\Rad{A}$ are symmetric. Then the symmetric interval matrix is defined as the subset of $\imace{A}$ formed by symmetric matrices, that is,
$$
\smace{A}=\{A\in\imace{A}\mmid A=A^T\}.
$$

Complex interval numbers can be represented in several forms \cite{AleHer1983,Nic1980}, among which the most frequently used are rectangular and circular forms. Given $\inum{a},\inum{b}\in\IR$, the rectangular form reads
$$
\inum{a}+i\inum{b}=\{a+ib\mmid a\in\inum{a},\,b\in\inum{b}\}.
$$
Given $c\in\C$ and $r\geq0$, the circular form \cite{Hen1971} reads
$$
\langle c,r\rangle=\{z\in\C\mmid |z-c|\leq r\}.
$$
While the rectangular form gives precise results for addition and subtraction, circular form is more convenient for multiplication and division. We will get use of circular complex numbers, which are also implemented in \textsf{Intlab}~\cite{Rum1999}, and we denote the set of such complex intervals by~$\IC$.

\section{Spectral decomposition of an interval matrix}

If a matrix $A\in\R^{n\times n}$ is diagonalizable, we can write it as $A=V\Lambda V^{-1}$, where $\Lambda\in\C^{n\times n}$ is a diagonal matrix with eigenvalues of $A$ on the diagonal, and $V\in\C^{n\times n}$ is a matrix having the corresponding eigenvectors as its columns.
If $A$ is symmetric, then eigenvalues and eigenvectors are real and $V$ can be constructed to be orthogonal.
Since spectral decomposition is not unique, we consider just one particular instance.

Given $\imace{A}\in\IR^{n\times n}$, our problem (called a spectral decomposition of an interval matrix) states
\begin{quotation}\sl\noindent
Find a diagonal matrix $\imace{\Lambda}\in\IC^{n\times n}$ and a matrix $\imace{V}\in\IC^{n\times n}$ such that for each $A\in\imace{A}$ there are $\Lambda\in\imace{\Lambda}$ and $V\in\imace{V}$ such that $A=V\Lambda V^{-1}$.
\end{quotation}
Once we have an enclosure $\imace{V}$, we can compute an enclosure of the set of inverses $\{V^{-1}\mmid V\in\imace{V}\}$ by methods from interval computation; see e.g.\ \cite{RohFar2011,Roh2011a}. We will use $\imace{V}^{-1}$ to denote an interval enclosure of the set of inverses computed by a certain method. In our implementation, we employ the standard function from Intlab~\cite{Rum1999}.

For a given symmetric interval matrix $\smace{A}\in\IR^{n\times n}$, the problem statement is analogous:
\begin{quotation}\sl\noindent
Find a diagonal matrix $\imace{\Lambda}\in\IR^{n\times n}$ and a matrix $\imace{Q}\in\IR^{n\times n}$ such that for each $A\in\smace{A}$ there are $\Lambda\in\imace{\Lambda}$ and orthogonal $Q\in\imace{Q}$ such that $A=Q\Lambda Q^T$.
\end{quotation}
To find a spectral decomposition of an interval matrix, we need to compute enclosures of eigenvalues and eigenvectors first. These issues are addressed in the next sections.

\subsection{Enclosing the eigenvalues}

For enclosing eigenvalues of interval matrices, we utilize the Bauer--Fike theorem \cite{GolLoa1996, HorJoh1985}. Therein, $\kappa_p({A})=\|{A}\|_p\|{A}^{-1}\|_p$ is the condition number and $\|\cdot\|_p$ is the induced matrix $p$-norm.

\begin{theorem}[Bauer--Fike, 1960]
Let $A,B\in\R^{n \times n}$ and suppose that $A$ is diagonalizable, that is, $V^{-1}AV=\diag(\lambda_1,\dots,\lambda_n)$ for some $V\in\C^{n \times n}$ and $\lambda_1,\dots,\lambda_n\in\C$. Then for every (complex) eigenvalue $\lambda(A+B)$ of $A+B$ there is some $i\in\seznam{n}$ such that
\begin{align*}
|\lambda(A+B)-\lambda_i|\leq 
  \kappa_p(V)\cdot\|B\|_p.
\end{align*}
\end{theorem}

We adapt this result for an interval matrix $\imace{A}$ following the approach used for a different purpose, see Hlad\'{\i}k et al.~\cite{HlaDan2010}. We choose $p=2$. Any $A\in\imace{A}$ can be written as $A=\Mid{A}+A'$, where $|A'|\leq\Rad{A}$. Since $\|A'\|_2\leq\|\Rad{A}\|_2$, we have, by the Bauer--Fike theorem, that for each complex eigenvalue $\lambda(A)$ there is some complex eigenvalue $\lambda_i(\Mid{A})$  such that
\begin{align*}
|\lambda(A)-\lambda_i(\Mid{A})|
\leq \kappa_2(V)\cdot\|\Rad{A}\|_2.
\end{align*}

Recall that by $\langle c,r\rangle=\{z\in\C\mmid |z-c|\leq r\}$ we denote the ball (disc) in the complex plane with center $c$ and radius~$r$. In particular, put 
$c_i:=\lambda_i(\Mid{A})$ and $r:=\kappa_2(V)\cdot\|\Rad{A}\|_2$.
The above reasoning says that for each $A\in\imace{A}$, every eigenvalue $\lambda(A)$  lies within one of the discs $\langle c_i,r\rangle$ for some $i\in\seznam{n}$.

For the purpose of construction of a spectral decomposition of $\imace{A}$ we need the following assumption:
\begin{quotation}\sl\noindent
\textbf{Assumption.}
The discs $\langle c_i,r\rangle$, $i\in\seznam{n}$, are disjoint.
\end{quotation}
Under this assumption, the matrix $\imace{\Lambda}$ can be simply constructed such that the discs $\langle c_i,r\rangle$, $i\in\seznam{n}$ are situated on the diagonal. When working with rectangular complex intervals, we just take $[\Re(c_i)-r,\Re(c_i)+r]+i[\Im(c_i)-r,\Im(c_i)+r]$ instead.

Notice that there are also other ways known to estimate eigenvalues of general interval matrices; for instance \cite{Kol2006,QiuMulFro2001,RohDei1992}.
In the symmetric case, there exist several methods for estimating eigenvalues; see e.g.~\cite{Her1992,HlaDan2011b,HlaDan2011c,LenHe2007}.
In our numerical experiments, we use a combination of the methods presented in~\cite{HlaDan2010}.

\subsection{Enclosing the eigenvectors}

Rohn \cite{Roh1993b} presented a simple (strongly polynomial) test for checking whether a real vector $x\in\R^n$ is an eigenvector of some $A\in\imace{A}$. This is, however, not suitable for computing an enclosure $\imace{V}$. Hence we propose a method based on a different idea. Alternative approaches were proposed, e.g., in \cite{Kol2006} for the general and in \cite{MiyOgi2010,Rum2010} for the symmetric case.

\paragraph{General case}
Let $A\in\imace{A}$ and $\lambda$ be one of its eigenvalues. From the Assumption we know that it is a simple eigenvalue. Thus the matrix $B:=A-\lambda I_n$ has rank $n-1$. 
Let $x$ be an eigenvector corresponding to $\lambda$, that is, $Bx=0$. Since $x\not=0$, there is $j\in\seznam{n}$ such that $x_j\not=0$, and without loss of generality we can assume that $x$ is normalized such that $x_j=1$. Thus we can rewrite equation $Bx=0$ into the form $\tilde{B}\tilde{x}=-B_{*j}$, where $B_{*j}$ denotes $j$-th column of matrix $B$, and $\tilde{B}$ denotes $B$ after removing the $j$-th column and similarly for~$\tilde{x}$. This system is solvable, but overdetermined, so we can omit one equation. Since $B$ has rank $n-1$, there is some equation $i$ that can be removed and the resulting system $B^{ij}\tilde{x}=-B^i_{*j}$ has nonsingular matrix~$B^{ij}$. 

This reasoning holds also the other way round. If $B^{ij}$ is nonsingular for some $i,j\in\seznam{n}$ and $\tilde{x}$ is the solution of $B^{ij}\tilde{x}=-B^i_{*j}$, then extending $\tilde{x}$ to the full length vector $x$ by inserting $1$ yields an eigenvector $x$ of~$A$.

The above discussion justifies the method for computing an enclosure of eigenvectors and consequently~$\imace{V}$. Algorithm~\ref{algEigvec} gives the pseudocode of this method. In step~\ref{stepAlgEigvecIls}, we need to compute an enclosure of the solution set of interval linear equations. In general, the entries are complex intervals. The solution set of complex interval systems has a complicated structure and is hard to characterize \cite{Hla2010f}. However, we can transform it to a real interval system of double size and apply standard methods. This transformation causes linear dependencies between the interval coefficients; we can relax these dependencies, which results in (hopefully slight) overestimation.
In our implementation, we call the \textsf{Intlab} function \texttt{verifylss}, ignoring the dependencies.

\begin{algorithm}[t]
\KwIn{Matrix $\imace{A}\in\IR^{n\times n}$ and enclosure of $k$-th eigenvalue $\inum{\lambda}\in\IR$}
\KwResult{Interval enclosure $\ivr{x}\in\IR^n$ of $k$-th eigenvectors of $\imace{A}$}
$\imace{B} \leftarrow \imace{A} - \inum{\lambda} I_{n}$;\\
\For{$i = 1, \dots, n$}{
\For{$j = 1, \dots, n$}{
compute an enclosure $\tilde{\ivr{x}}$ to the interval system $\imace{B}^{ij}\tilde{x} = -\imace{B}^{j}_{*j}$;\label{stepAlgEigvecIls}\\
\If {succeeded}{
extend $\tilde{\ivr{x}}$ to $\ivr{x}$ by inserting the component~1;\\
\KwRet{$\ivr{x}$}\\ 
}
}
}
\KwRet{\quo{algorithm failed to compute $\ivr{x}$}}\\ 
\caption{Interval enclosure of eigenvectors\label{algEigvec}}
\end{algorithm}

\paragraph{Symmetric case}

Due to our Assumption, the computed eigenvectors are always orthogonal. However, they need not be unit. Therefore, we normalize them even in the interval enclosure form. Let $\ivr{v}$ be $i$-th column of $\imace{Q}$. Thus, we divide this column by the interval $\|\ivr{v}\|_2=\{\|v\|\mmid v\in\ivr{v}\}$. This interval of norms of all vectors in $\ivr{v}$ can be easily and exactly computed up to numerical accuracy. We perform this procedure for each column of $\imace{Q}$ so that we can be sure that each set of orthonormal eigenvectors is included in~$\imace{Q}$. As a consequence, the spectral decomposition of $\smace{A}$ has the desired form of $\smace{A}\subseteq \imace{Q}\imace{\Lambda}\imace{Q}^T$.

If we fail to compute the eigenvectors by the above procedure, we can set all entries of $\imace{Q}$ to be $[-1,1]$. Since the entries of orthogonal matrices lie in this interval, we have a guaranteed enclosure. Even though the enclosure may seem very rough, for wide input intervals it pays off. In addition, the Assumption is not needed.

In step~\ref{stepAlgEigvecIls}, we compute an enclosure of the solution set of interval linear equations. In this case the entries are real intervals, however, there are dependencies between the interval entries caused by symmetry of $\smace{A}$. In our implementation, we relax them, which results in mild overestimation of the eigenvectors. Notice that there are known methods that can take into account such dependencies to some extent; see \cite{Rum2010,Ska2018,Kol2004c,Pop2019au,SkaHla2017b} among others.

\begin{example}
Let $\smace{A}$ be a diagonal matrix with interval entries $\inum{d}_1,\dots,\inum{d}_n\in\IR$. The best spectral decomposition is obviously $\imace{Q}=I_n$, $\imace{\Lambda}=\smace{A}$. As long as we employ a sharp method for enclosing eigenvalues (which is the case of \cite{HlaDan2010}), then indeed we obtain $\imace{\Lambda}=\smace{A}$. Under the Assumption, we also compute tight envelopes of the eigenvectors. For the $i$-th interval eigenvalue $\inum{\lambda}$, the interval matrix $\smace{A}-\inum{\lambda}I_n$ is diagonal and only the $i$-th diagonal entry contains the zero. Therefore our procedure for computing eigenvectors succeeds and returns $\ivr{v}=e_i$, the $i$-th canonical unit vector. Thus, rather surprisingly, the method is able to compute tight enclosure $\imace{Q}=I_n$ for eigenvectors.
\end{example}

\section{Numerical study: Computing powers of interval matrices}



Eigenvalue decomposition is a very important technique in numerical algorithms. 
Likewise eigenvalue decomposition of an interval matrix can help to solve problems in the discipline of interval computation. This section is devoted to analysis of applying spectral decomposition of a matrix $\imace{A}$ to computing the powers of $\imace{A}$. Computing powers is important not only in the real case \cite[chap.~17]{Hig1996}, but also in the interval case, e.g., for calculation of the exponential of interval matrices \cite{GolNeu2014,OppMich1988}.

The $k$-th power of an interval matrix $\imace{A}\in\IR^{n\times n}$ is defined as a set of all $k$-th powers of all realizations, that is,
\begin{align*}
\imace{A}^k:=\{A^k\mmid A\in\imace{A}\}.
\end{align*}
Since this is not necessarily an interval matrix, we will be content with its interval hull
\begin{align*}
\hulli{\imace{A}^k}:=\ihull\{A^k\mmid A\in\imace{A}\}
\end{align*}
instead. 
Whereas computing $[\imace{A}^2]$ is easy by interval arithmetic, computing the cube $[\imace{A}^3]$ is an NP-hard problem~\cite{KoshKre2005}. We will show that this is the case even for symmetric interval matrices.

\begin{proposition}
Computing $[(\smace{A})^3]$  is NP-hard.
\end{proposition}

\begin{proof}
Let $C\in\R^{n\times n}$ be symmetric and $\ivr{x}\in\IR^n$. It is known \cite{KreLak1998,Vav1991} that computing $\max_{x\in\ivr{x}} x^TCx$ is NP-hard. Consider the symmetric interval matrix
$$
\smace{A}:=\begin{pmatrix}C&\ivr{x}\\\ivr{x}^T&0\end{pmatrix}^S.
$$
For each $A\in\smace{A}$ we have 
$$
A^2=\begin{pmatrix}C^2+xx^T&Cx\\x^TC&x^Tx\end{pmatrix},\quad
A^3=\begin{pmatrix}C^3+Cxx^T+xx^TC&C^2x+xx^Tx\\x^TC^2+x^Txx^T&x^TCx\end{pmatrix}.
$$
Therefore computing the right end-point of $[(\smace{A})^3]_{n+1,n+1}$ is NP-hard.
\end{proof}

This property indicates that computation of the higher powers is even a more complicated problem. The core problem is in the number of interval arithmetic operations with multiple occurrences of variables. A natural way to reduce a possible overestimation is a reduction of computations with this property.

There are two basic approaches in hand for computing an enclosure of $\hulli{\imace{A}^k}$:
\begin{enumerate}
\item
\emph{Binary exponentiation.}
This method, classically applied for powers of real matrices, computes the $k$-th power similarly as for the real matrix via approximately $\log(k)$ consecutive squaring of matrices according to the binary representation of the exponent $k$. We just replace floating point arithmetic by interval arithmetic. 
\item
\emph{Spectral decomposition.}
The second method is based on the above discussed spectral decomposition. Let $\imace{V}$ and $\imace{\Lambda}$ be the interval matrices computed from a spectral decomposition of $\imace{A}$. Then $\hulli{\imace{A}^k}\subseteq \imace{V}\imace{\Lambda}^k\imace{V}^{-1}$, where $\imace{\Lambda}^k$ is simply the matrix $\imace{\Lambda}$, in which we take the $k$-th power of the diagonal entries calculated by interval arithmetic.
\end{enumerate}

In the following we compare numerically these two approaches. The calculations were done in \textsf{Matlab} 9.3.0 (R2017b). For interval arithmetic and other interval functions we employ the interval package \textsf{Intlab} 11\cite{Rum1999} and the interval library Lime~\cite{Lime} (e.g., for computation of eigenvalues for symmetric matrices).
Notice that all calculations were done in a verified way \cite{Rum2010}, so that the resulting enclosures are numerically guaranteed; to this end, we utilized the verification software from \textsf{Intlab}.
All computations were done on a server with AMD EPYC 7301 16-Core Processor and 128 GB RAM although each computation was limited to single CPU only. The server was operated with Debian GNU/Linux 9.8 (stretch).

As we mentioned earlier, for a symmetric interval matrix $\smace{A}$ we can simply enclose the entries of $\imace{Q}$ by the interval $[-1,1]$. In this way, the enclosure $\hulli{(\smace{A})^k}\subseteq \imace{Q}\imace{\Lambda}^k\imace{Q}^T$ is the interval matrix with identical entries, yielding:

\begin{observation}\label{obser:powersym}
We have 
$\hulli{(\smace{A})^k}\subseteq \imace{H}$, where $\imace{H}_{ij}=[-h,h]$ and $h=\sum_{i=1}^n\Mag(\inum{\Lambda}_{ii})^k$.
\end{observation}

\subsection{General testing pipeline}

Our approach starts with generating a set of random matrices of a predefined type for consecutive testing. Let $n$ be a fixed size of tested matrices. Let $N_t$ denote the number of runs, which in our case is set-up at the value $1000$. Random generators expect three parameters: the dimension of the generated matrices ($n$), the real center value ($c$) and the radius~($r$). Depending on the matrix class, appropriate random matrices are generated; see section~\ref{ss:test:genint} for general interval matrices, section~\ref{ss:test:sym} for symmetric interval matrices and section~\ref{ss:test:circ} for circulant interval matrices.

Let $k\in \{k_1, k_2, \ldots, k_M\}$ be the exponent of the test set and $t \in \{1, 2, \ldots, N_t\}$ the index of a current run. For the generated interval matrix $\imace{A}_t$ we apply the the following steps:
\begin{enumerate}
\item Compute $k$-th power $\imace{B}_t$ of the matrix $\imace{A}_t$ using the \emph{binary exponentiation}.
\item Compute $k$-th power $\imace{C}_t$ of the matrix $\imace{A}_t$ using the \emph{spectral decomposition}.
\item Compare interval enclosures $\imace{B}_t$ and $\imace{C}_t$.
\end{enumerate}
For comparison of interval enclosures we utilize the sum of all radii, i.e., the quantity
\begin{equation*}
R(\imace{B}_t) = \sum_{i,j} (\imace{B}_t^{\Delta})_{i,j}.
\end{equation*}
While we assume that binary exponentiation represents a natural way to compute powers of interval matrices, we define the following ratio to compare this method to the spectral counterpart,
\begin{equation}\label{eq:rate-sum-radii}
\rho_t = \frac{R(\imace{C}_t)}{R(\imace{B}_t)}.
\end{equation}

The final step of the procedure is to produce an average of this rate for all tested matrices. This is not always possible for the whole tested set since the computation of a spectral decomposition may fail for a few of reasons, e.g. the most sensitive part of exponentiation is computation of eigenvectors. If computation of $\imace{C}_t$ fails, we throw away the whole pair $\imace{B}_t$, $\imace{C}_t$. For this reason we define $T_r = \{t_1, t_2, \ldots, t_{N_t'}\}$ to be the set of indices of successful computations where $N_t' \leq N_t$. For such a reduced set of instances we compute the average value $\hat{\rho}$ and the median $\tilde{\rho}$ of characteristic~\eqref{eq:rate-sum-radii}.


\subsection{Powers of general interval matrices}\label{ss:test:genint}

The first type of tested matrices is represented by general interval matrices. For a fixed $t \in \{1,2,\ldots, N_t\}$ let $G_t$ and $G_t'$ denote $n \times n$ matrices with elements chosen randomly uniformly from intervals $(-1,1)$ and $(0,1)$, respectively. For a given center value $c$ and a chosen radius $r$ we define a random interval matrix as $\imace{A}^R_t = [A_{c,t} - A^{\Delta}_t,A_{c,t} + A^{\Delta}_t]$, where $A_{c,t} = c\cdot G_t$ and $A_t^{\Delta} = r \cdot G_t'$. Each such matrix is further normalized using the magnitude of its interval norm, i.e., 
\begin{equation*}
	\imace{A}_t = \imace{A}^R_t / \|\imace{A}^R_t\|_M,
\end{equation*}
where $\|\cdot\|_M$ represents the right end-point of the interval enclosure of the spectral norm values computed by \textsf{Intlab}. The results of our testing can be explored in Figure~\ref{fig:rnd_power_accur}.

\begin{figure}[t]
     \subfloat[$n = 5$, $r=0.001$, starting at $k=15$.]{%
       \includegraphics[width=5.5cm]{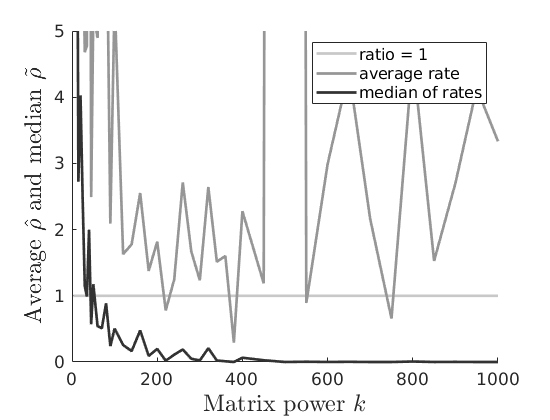}
       \label{fig:rnd_power_accur:first}
     }
     \subfloat[$n = 5$ and $r=0.01$]{%
       \includegraphics[width=5.5cm]{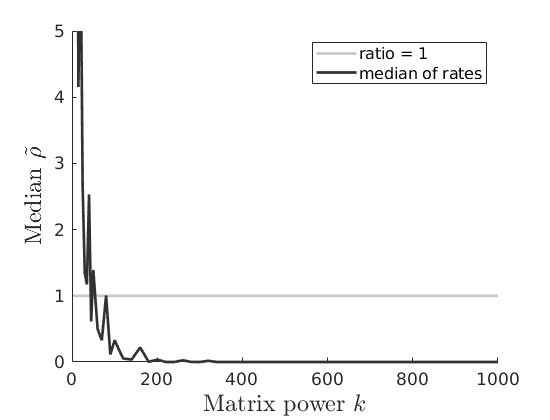}
     }
     \subfloat[$n = 10$ and $r=0.001$]{%
       \includegraphics[width=5.5cm]{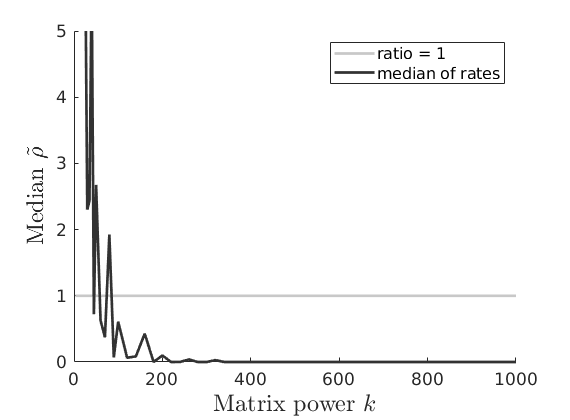}
     }
  \caption{(General case) Comparison of the spectral and the binary method for computing interval matrix powers using the ratio~\eqref{eq:rate-sum-radii}. The results were obtained for a testing set of $N_t = 1000$ square matrices of sizes $n \in \{5, 10\}$ with the center $c = 10$ and the radius $r \in \{0.01, 0.001\}$. The light gray line shows the value $1$, representing similar accuracy of the methods.}
  \label{fig:rnd_power_accur}
\end{figure}

For general random matrices with the radius $r=0.001$ and the dimension $n=5$, we successfully compute almost $90\%$ of all powers, where the majority of remaining cases were due to problems with inversion of the eigenvector matrices (caused by singularity of $\imace{V}$). The number of cases where the discs enclosing eigenvalues are intersecting, violating thus the Assumption, remains below $0.5\%$. Moreover, we can see in Figure~\ref{fig:rnd_power_accur:first} that the average values are significantly influenced by numerical extremes. For this reason we have chosen median as a well suited measure to compare these methods. For two other cases presented on Figure~\ref{fig:rnd_power_accur} we enlarge either the radius to $r=0.01$ or the dimension to $n=10$. For both extensions the median values behave roughly the same, as shown in Figure~\ref{fig:rnd_power_accur}. The main difference is in the fraction of failures for these cases -- generally a problem with inversion of $\imace{V}$ remains major with a slight grow of the effect of intersecting discs. When enlarging the radius, the number of problems with inversion of $\imace{V}$ achieves almost $60\%$, and the number of cases where the discs were intersecting reached $4\%$. 

\subsection{Powers of symmetric interval matrices}\label{ss:test:sym}

Symmetric interval matrices we generated as follows. For a fixed $t \in \{1,2,\ldots, N_t\}$ let $G^S_t$ and $G'^S_t$ denote $n \times n$ symmetric matrices with elements chosen randomly uniformly from $(-1,1)$ and $(0,1)$, respectively. The symmetry is simply achieved by randomly generating the diagonal and the above diagonal entries only and reusing them for the remaining ones. For a given midpoint value $c$ and a chosen radius $r$ we define a random interval matrix as $\imace{A}^{SR}_t = [A^S_{c,t} - A^{S,\Delta}_t, A^S_{c,t} + A^{S,\Delta}_t]$, where $A^S_{c,t} = c \cdot G^S_t$ and $A_t^{S,\Delta} = r \cdot G'^S_t$.  To reduce possibility of exponential grow of matrix power values, we again normalize each matrix as follows $\imace{A}^{S}_t = \imace{A}^{SR}_t / \|\imace{A}^{SR}_t\|_M$.

\begin{figure}[ht]
     \subfloat[$n = 5$ and $r=0.001$]{%
       \includegraphics[width=5.5cm]{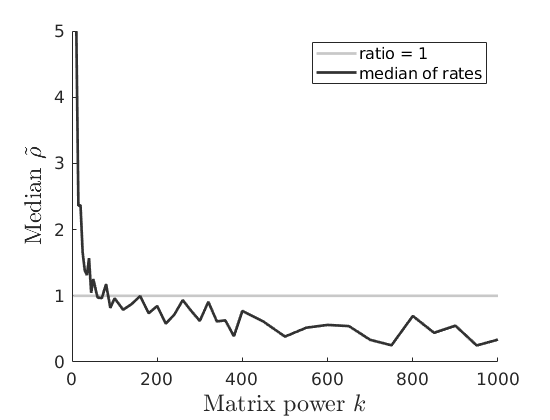}
     }
     \subfloat[$n = 5$ and $r=0.1$]{%
       \includegraphics[width=5.5cm]{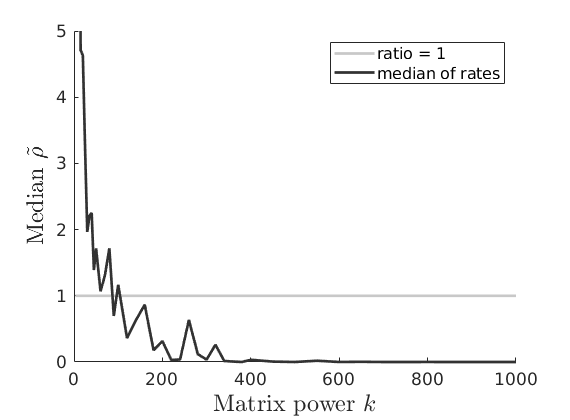}
     }
     \subfloat[$n = 20$ and $r=0.001$]{%
       \includegraphics[width=5.5cm]{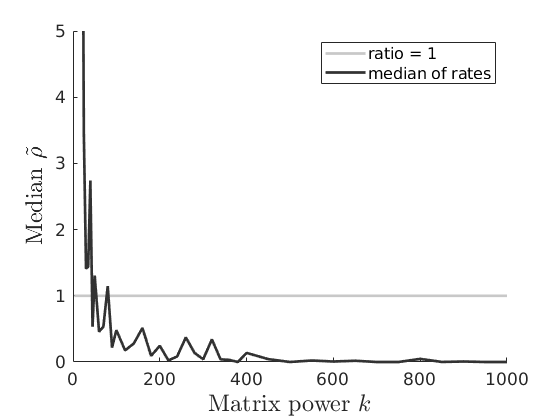}
     }
  \caption{(Symmetric matrices) Comparison of the spectral and the binary method for computing interval matrix powers using the ratio~\eqref{eq:rate-sum-radii}. The results were obtained for a testing set of $N_t = 1000$ square matrices of sizes $n \in \{5, 20\}$ with the center $c = 10$ and the radius $r \in \{0.01, 0.001\}$. The light gray line shows the value $1$, representing similar accuracy of the methods.}
  \label{fig:sym_power_accur}
\end{figure}
For these randomly generated symmetric matrices, the results are shown in Figure~\ref{fig:sym_power_accur}. In this case, regardless of size and radius, the resulting median behaves the same. In contrast to computations on general matrices there were no failures and thus all powers were successful in computation. Of course, when powers of symmetric matrices generated in this section are computed by the general method, the results are significantly worse. But due to the space limitation we do not include the output of the corresponding computations here.

\subsection{Powers of special interval matrices}\label{ss:test:circ}

In this section we describe powers for special types of matrices making use of their specific properties. 

\paragraph{Circulant matrices}
The first class of special matrices are circulant matrices. Given a vector $c = \{c_0, c_1, \ldots, c_{n-1}\}$, the corresponding circulant matrix reads as
\begin{equation}\label{eq:circulant}
C = 
\left(\begin{matrix}
    c_0 		& c_1	 		& c_2		& \hdots		& c_{n-2}		& c_{n-1}	\\
    c_{n-1}		& c_0	 		& c_1	 	& c_2			& \hdots		& c_{n-2}	\\
    \vdots		& \ddots		& \ddots	& \ddots 		& \ddots 		& \vdots	\\
    c_2			& c_3			& \ddots	& c_{n-1} 		& c_0 			& c_1		\\
    c_1	 		& c_2			& \hdots	& c_{n-2}		& c_{n-1}		& c_0		\\
\end{matrix}\right).
\end{equation}

It is well known that normalized eigenvectors of a circulant matrix are defined for any $j = 0, 1, \ldots, n -1$ as
\begin{equation}
v_j = \frac{1}{\sqrt{n}}(1,\omega_j, \omega_j^2, \ldots, \omega_j^{n-1}),
\end{equation}
where $\omega_j = \exp\left(i\frac{2\pi j}{n}\right)$ are $n$ different roots of unity. For any $j = 0, 1, \ldots, n -1$, the corresponding eigenvalue is then as follows
\begin{equation}
\lambda_j = c_1\omega_j^{n-1} + c_2\omega_j^{n-2} + \ldots + c_{n-2}\omega_j^2 + c_{n-1}\omega_j + c_0.
\end{equation}

To compute the power of an interval circulant matrix, where real elements are exchanged by intervals, we can compute the powers using the spectral method via utilization of the above mentioned exact form of eigenvectors and eigenvalues using interval arithmetic. Having both eigenvectors and eigenvalues computed we can define the spectral decomposition and thus compute the power. To do so, we need to generate a set of random circulant matrices.

For a fixed $t \in \{1,2,\ldots, N_t\}$ let $g_t$ and $g'_t$ denote $n$-dimensional vectors with elements chosen randomly uniformly from $(-1,1)$ and $(0,1)$, respectively. For a given center value $c$ and a chosen radius $r$ we determine two transformed vectors $a_{c,t} = c \cdot g_t$ and $a_t^{\Delta} = r \cdot g_t'$. To define a random interval circulant matrix we simply distribute elements $a_{c,t}$ and $a_t^{\Delta}$ into circulant position obtaining thus circulant matrices $A^C_{c,t}$ and $A_t^{C,\Delta}$. 
A random interval circulant matrix is then given by $\imace{A}^{C}_t = [A^C_{c,t} - A^{C,\Delta}_t,A^C_{c,t} + A^{C,\Delta}_t]$ and similarly to previous cases we introduce normalization $\imace{A}^{C}_t = \imace{A}^{C}_t / \|\imace{A}^C_t\|_M$. We can see representative results of testing circulant matrices in Figure~\ref{fig:circ_power_accur}.

\begin{figure}[ht]
\centering{
     \subfloat[Circulant matrix with $d = 5$ and $r=0.001$]{%
       \includegraphics[width=5.5cm]{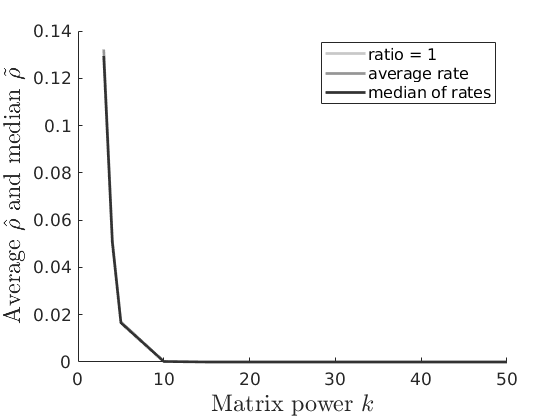}
       \label{fig:circ_power_accur}
     }
     \hspace{0.5cm}
     \subfloat[Symmetric matrix with $d = 5$ and $r=1.0$]{%
       \includegraphics[width=5.5cm]{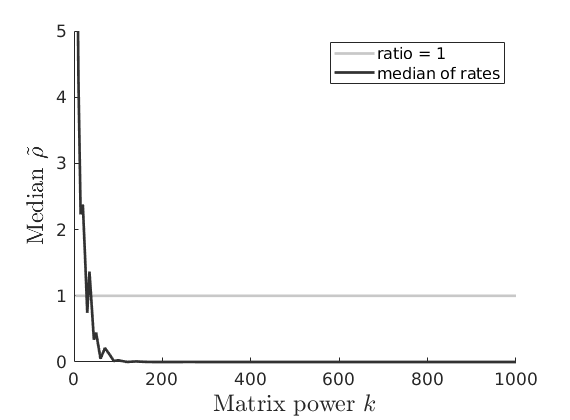}
       \label{fig:fat_sym_power}
     }
     }
  \caption{(Special matrices) Comparison of the spectral and the binary method when computing interval matrix powers using the ratio~\eqref{eq:rate-sum-radii}. (Left) The results shown were obtained for a testing set of $N_t = 1000$ random square circulant matrices of sizes $n = 5$ with the center $c = 10$ and the radius $r = 0.001$. (Right) The results shown were obtained for a testing set of $N_t = 1000$ random square symmetric matrices of sizes $n = 5$ with the center $c = 10$ and the radius $r = 1.0$. The light gray line shows the value $1$, representing similar accuracy of the methods.}
\end{figure}

We can see that in this ideal case we can compute the powers very effectively -- note that in this case we exceptionally present average values as well since these are equal to median values.

\paragraph{Symmetric interval matrices with large intervals}
For intervals with large radius the Assumption is often violated and so consequent computations of $\imace{V}$ fail. Therefore for wide intervals we use another approach based on Observation~\ref{obser:powersym}. This observation, translated to the spectral method, simply says that we can use a matrix of interval $[-1,1]$ as a matrix of eigenvectos for our previously computed symmetric eigenvalues. This approach enables us to avoid the application of Algorithm~\ref{algEigvec}.

The results of matrix power computations for random symmetric interval matrices with large radii are displayed in Figure~\ref{fig:fat_sym_power}. We can see that using this simple approximation of eigenvectors is relatively successful for large radii. This, however, holds only for a large radius. Reducing the radius to the original value of $r=0.001$ makes the spectral method quite ineffective.

%
%
%
%

\subsection{Discussion of the results}

The results have one (rather obvious) feature in common: the larger the exponent in the matrix power, the more efficient the spectral decomposition method is. For small powers, the binary exponentiation naturally must win, but for higher powers, the spectral decomposition approach comes into play. The change-point varies with methods as well as with the radius values and matrix sizes. For general random matrices it is roughly at $k=50$ and for symmetric around $k=80$. Notice that we call quite basic functions for estimating the eigenvalues and for solving interval linear systems. An application of more advanced methods and utilizing the structure of the problem may further decrease the change-point and improve other properties of the methods.

The results of computations for general random interval matrices seem to be slightly better than for symmetric matrices. The disadvantage for general matrices is, however, the inability of the method to compute the results for all random inputs. The ratio of successful computations decreases with the increase of the radius and dimension of generated matrices. This can end in the inability to compute the power for most of the inputs. Computing powers of random symmetric interval matrices almost does not suffer from this problem; in our experiments all instances were successfully computed. In this case, there is a problem with extreme values of powers possibly caused by a combination of numerical error and some problems in eigenvalue and eigenvector estimation. It is worth mentioning that we normalized the interval matrices in order to avoid divergence of computing the power. For general interval matrices their powers might quickly grow to infinity.

We presented also some results for special classes of matrices. First, for interval circulant matrices, the spectral decomposition method is always better or the same as the binary exponentiation. This is a motivation to improve the spectral method even for exponentiation -- improvements of eigenvalue and eigenvector estimations can be beneficial for computing the powers. Another example of a specific structure is a symmetric interval matrix with a large radius. Such a matrix can be problematic for computing its power due to a potentially long computation time and additional problems with numerical stability. We showed, however, how to handle such a specific matrices using a simple observation about its spectral decomposition. Both these examples underline the well known fact that utilizing the specific properties of special matrices has a high potential of obtaining better results in the orders of magnitude.

Time complexity of matrix power computation using the spectral decomposition is practically independent of exponent values. In contrast, time complexity of the binary method grows with the exponent. Generally, time complexity of the spectral method is higher than time complexity of the binary method. Thanks to the growth of time complexity of the binary method, the average ratio of the spectral method and the binary method computation time is decreasing, and thus the spectral method is becoming less inefficient. Roughly speaking, for general random interval matrices this ratio is about $10$ and for symmetric matrices is has value around $100$. For symmetric matrices this higher value is due to more complicated (but more accurate) computation of eigenvalues.


\section{Conclusions}

We presented an algorithm for computing an enclosure of the eigenvalue decomposition of an interval matrix. Both the enclosures of eigenvalues and eigenvectors are important for many applications, but the whole spectral decomposition enables further operations with interval matrices. As a basic example, we considered the problem of computing powers of interval matrices. It turned out that for a general or symmetric interval matrix and small exponents, the computation of matrix powers using the spectral decomposition is less effective than the binary exponentiation. Nevertheless, as the exponent increases, the spectral decomposition approach 
is more efficient.
Moreover, as long as we have some additional information (e.g., for structured matrices), effectiveness of the approach might substantially rise.
The results are particularly convincing for circulant interval matrices, having a specific form of eigenvalues and eigenvectors. The conclusion from this computation is twofold. First, it encourages further investigation of special matrices as a promising way of research. Second, it suggests that potential improvement of eigenvalues and eigenvectors estimations might help the spectral method to become broader applicable for computing matrix powers or other problems.

There is some space for further improvements; we mention two promising directions. First, in the recent years, there was a development of the so called parameterized solution \cite{Kol2014b,Ska2018,SkaHla2017b}. It is an enclosure in a generalized form and provides tighter enclosures than interval boxes, in particular when the enclosure is subject to subsequent calculations. Thus a parameterized form of eigenvector estimation could provide tighter results. Second, the so called affine arithmetic \cite{Ska2018,SkaHla2017a,StoFig1997} is a generalized interval arithmetic that is able to keep certain correlations between interval quantities and produces tighter enclosures in general. Incorporating affine arithmetic thus may lead to further tightening of spectral decomposition.

\subsubsection*{Acknowledgments.} 
The authors were supported by the Czech Science Foundation Grant P403-18-04735S.

    \section*{References}
    \bibliography{lit-si}

    \end{document}